\documentclass{article}
\usepackage{amssymb, amsfonts, amsmath, amsthm}
\usepackage{enumerate}
\usepackage{caption}
\usepackage[colorlinks=true,
citecolor=black,
linkcolor=black,
anchorcolor=black,
filecolor=black,
menucolor=black,
urlcolor=black
]{hyperref}
\usepackage[normalem]{ulem}
\usepackage{cancel}
\usepackage{array}
\usepackage{longtable}
\usepackage{tikz}   
\usepackage{epsfig,lpic,wrapfig}
\usepackage{bm}
\usepackage{multicol}

\usepackage[utf8]{inputenc}
\usepackage[T1,T2A]{fontenc}
\usepackage[russian,
german,
english]{babel}

\usepackage{csquotes}

\def\thetitle{Two moons in a puddle}
\def\theauthors{Berk Ceylan}
\hypersetup{pdftitle={\thetitle},
pdfauthor={\theauthors}}

\newcommand{\Addresses}{{\bigskip\footnotesize

\medskip

\noindent   Berk Ceylan, 
\par\nopagebreak
 \textsc{}
  \par\nopagebreak
  \textit{Email}: \texttt{ceylan.berk@outlook.com}
  
}}

\def\parbf#1{\medskip\noindent{\bf #1}}

\theoremstyle{thm}

\theoremstyle{theorem}
\newtheorem{theorem}{Theorem}
\newtheorem{lemma}{Lemma}

\newcounter{thm}[section]

\def\claim#1{\par\medskip\noindent\refstepcounter{thm}\hbox{\bf\boldmath #1.}
\it\ 
}
\def\endclaim{
\par\medskip}

\begin{document}

\title{\thetitle}
\author{\theauthors}

\date{}
\maketitle

\parbf{Introduction:} 
The moon in a puddle theorem is one of the simplest local to global theorem in differential geometry. It states that a simple closed plane curve $\gamma$ with curvature at most 1 contains a unit disk inside. The theorem was originally proved by German Pestov and Vladimir Ionin \cite{pestovio}; see also \cite{Petrunin_2022}. In this note we discuss an analogous question for two open discs; in particular the following open question:
\medskip

    \textbf{Open Question.} \emph{Is it true that any simple closed plane curve with curvature at most 1 and length at least $4\pi$ surrounds two disjoint open unit disks?}
    \medskip

    The open question is interesting since there are many equality cases; where the length is exactly $4\pi$, but any two surrounded discs touch each other. We first prove our main result then give examples of equality cases. Also note that by pulling the bordered curve on the figure 1 to right and left one can see there is no three-moon analog of the open question.

\begin{center}
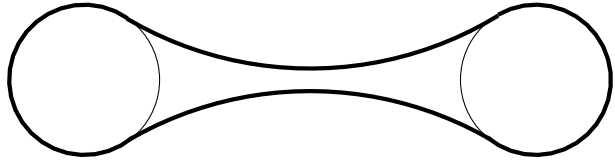

    \begin{tikzpicture}
        \coordinate (A) at (-3,0);
        \coordinate (B) at (3,0);
        \coordinate (C) at (0,4.8); 
        \coordinate (D) at (0,-4.8);

        \draw (A) circle (1);
        \draw (B) circle (1);

        \draw [black,ultra thick,domain=239.5:301.4] plot ({(4.8)*cos(\x)}, {(4.95)+(4.8)*sin(\x)});
        \draw [black,ultra thick,domain=60:120] plot ({(4.8)*cos(\x)}, {-(4.95)+(4.8)*sin(\x)});

        \draw [black,ultra thick,domain=55:310] plot ({-3+cos(\x)}, {sin(\x)});
        \draw [black,ultra thick,domain=120:-130] plot ({3+cos(\x)}, {sin(\x)});
        
    \end{tikzpicture}
\end{center}
\captionof{figure}{\emph{A curve that can only contain two disjoint unit circles.}}

\medskip

    \medskip

    \medskip
    \parbf{Two moons in a puddle:} The \textbf{incircle} at a point of a closed curve is the maximal circle which lies inside the curve and touches the chosen point. We denote the incircle at a point $p$ by $C_p$. For the proof of the main result we need the following adaptation of the key lemma in \cite{Petrunin_2022}.

    \begin{lemma}
        Let $\gamma$ be a simple closed plane curve and $p$ be a point on it. Assume that $\gamma_1$ is an arc of $\gamma$ with only the end points touching $C_p$. Then there is a point $q$ of $\gamma_1$ such that the osculating circle at $q$ supports $\gamma$ from inside. Moreover, any supporting osculating circles for two such arcs can only intersect inside $C_p$.
    \end{lemma}
    \begin{proof}
        Assume that at the points of $\gamma_1$, osculating circles does not support $\gamma$ from inside. Suppose that $q_1$ is the midpoint of $\gamma_1$. Then the incircle $C_{q_1}$ has to touch at least one other point of $\gamma$, otherwise we could enlarge it since its curvature is greater than $\gamma$ at $q_1$. Pick the first one clock-wise, say $r_1$. Observe that $r_1$ lies on $\gamma_1$. Then we get a sub-arc of $\gamma_1$ between $q_1$ and $r_1$, say $\gamma_2$. Denote the two arcs of $C_{q_1}$ between $q_1$ and $r_1$ by $\sigma_1$ and $\sigma_2$.
        We can assume that concatenation of $\gamma_2$ and $\sigma_1$ contains $\sigma_2$ inside. 
        
             Denote by $q_2$ the midpoint of $\gamma_2$. Observe that $C_{q_2}$ cannot touch $\sigma_1$. Then $C_{q_2}$ touches another point of $\gamma_2$. So we get a sub-arc $\gamma_3$ of $\gamma_2$ with 
            \begin{equation}
                length \gamma_3 < \dfrac{1}{2} length \gamma_2.
            \end{equation}

            Continuing this way, we get a sequence of nested arcs of $\gamma_1$ with length going to 0. Hence their intersection consists of a point $q_{\infty}$. Then the incircle  $C_{q_{\infty}}$ has to touch another point $r_{\infty}$. By the argument above $r_{\infty}$ lies on $\gamma_n$ for any $n$. Thus $r_{\infty}=q_{\infty}$, a contradiction. Thus there is a point of $\gamma_1$ with osculating circle supporting $\gamma$ from inside.

            We leave the proof of the last statement to reader.
    \end{proof}
    \begin{center}
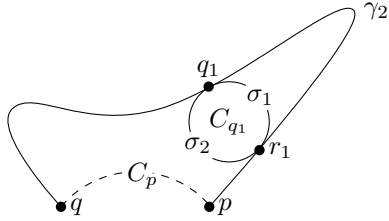

    \begin{tikzpicture}[scale=1.3]
    \draw[smooth,tension=0.8] plot coordinates{(1,1) (0.5,2) (2,2) (4,3) (2.5169,1)};

    \node[right] at (2.5169,1) {$p$};
    \fill[black]  (2.5169,1) circle(1.5pt);
    \node[right] at (1,1) {$q$};
    \fill[black]  (1,1) circle(1.5pt);

    \draw (2.72, 1.86928) circle (0.41);
    \node at (2.72, 1.86928) {$C_{q_1}$};

    \fill[black]  (2.51, 2.23) circle(1.5pt);
    \fill[black]  (3.03,1.58) circle(1.5pt);
    \node[above] at (2.51, 2.23) {$q_1$};
    \node[right] at (3.03,1.58) {$r_1$};


    \fill[white]  (3.04, 2.15) circle(5pt);
    \node at (3.04, 2.12) {$\sigma_1$};

    \fill[white]  (2.4,1.6) circle(5.2pt);
    \node at (2.4,1.63) {$\sigma_2$};
    \coordinate (A) at (1,1);
    \draw[dashed] (A) arc(140:40:1);

    \node[right] at (4,3) {$\gamma_2$};

    \fill[white]  (1.8,1.32) circle(4pt);
    \node at (1.836,1.32) {$C_p$};
    
  \end{tikzpicture}
\end{center}
\captionof{figure}{\emph{Sketch of the lemma.}}
       
\medskip
The reader can take hint from the figure 3 for the proof of the theorem.
        \begin{theorem}
    Let $\gamma$ be a simple closed plane curve with curvature at most 1 and diameter at least 4, then there are two disjoint open unit disks inside $\gamma$.
\end{theorem}

\begin{proof}
    Let $p_1,p_2$ be points on $\gamma$ with maximal distance. If both of the incircles $C_{p_1}$ and $C_{p_2}$ have radius at least 1 then by possibly shrinking them we get the result. So assume $C_{p_1}$ has radius less than 1. Then $C_{p_1}$ has to touch another point of $\gamma$, say $q$. Denote the two arcs of $\gamma$ between $p_1$ and $q$ by $\gamma_1$, $\gamma_2$ and, two arcs of $C_{p_1}$ between $p$ and $q$ by $\sigma_1$, $\sigma_2$. Assume concatenation of $\gamma_i$ with $\sigma_i$ contains $C_{p_i}$ inside. Then by curvature constraint $\gamma_1$ has points outside $C_{p_1}$. Hence by lemma, one of the osculating circles at points of $\gamma_1$ outside $C_{p_1}$, supports $\gamma$ from inside. Note that osculating circles have radius at least 1.  Note that supporting osculating circle can only intersect with $C_{p_2}$ inside $C_{p_1}$, and if they do intersect then $C_{p_2}$ has radius more than 1, thus by shrinking it we get the result. If they do not intersect we apply the same argument to $C_{p_2}$ to get two disjoint open unit disks inside $\gamma$. 
\end{proof} 

 \begin{center}
     \begin{tikzpicture}[scale=0.8]
         \draw [black, rounded corners, domain=180:-112.6] plot ({3*cos(\x)}, {3*sin(\x)});
         \draw [black, domain=180:360] plot ({(-2)+cos(\x)}, {sin(\x)});
         \draw [black,domain=66.5:246.5] plot ({(-0.75)+cos(\x)}, {(-1.85)+sin(\x)});
         \draw [black,rounded corners,domain=-127.5:194] plot ({(0.655)+(1.7)*cos(\x)}, {(0.418)+(1.7)*sin(\x)});

          \fill[black]  (1.7,-2.45) circle(1.5pt);
          \fill[black]  (-1.7,2.45) circle(1.5pt);
          \node[below] at  (1.7,-2.5) {$p_1$};
          \node[above] at (-1.7,2.5) {$p_2$};
           \fill[black]  (1.1,-1.22) circle(1.5pt);
           \node[above] at (1.1,-1.22) {$q$};

          \draw (1.29,-1.88) circle (0.7);
          \draw (-1.26,1.89) circle (0.71);
          \node at (1.29,-1.88) {$C_{p_1}$};
          \node at (-1.26,1.89) {$C_{p_2}$};

          \fill[white]  (-1.7,-1.5) circle(5pt);
          \node at (-1.7,-1.5) {$\gamma_1$};
          
            \fill[white]  (1.92,-1.60) circle(5pt);
            \node at (1.94,-1.62) {$\sigma_1$};

            \fill[white]  (0.64,-2.2) circle(5pt);
            \node at (0.64,-2.17) {$\sigma_2$};

     \end{tikzpicture}
 \end{center}
 \captionof{figure}{\emph{The idea is to apply lemma to $\gamma_1$ and its analogous arc for $C_{p_2}$.}}
 \medskip
 
    Note that if two disjoint open unit disks lie inside a curve as above, then its diameter is at least 4. So, this is actually an if and only if condition. From the theorem, our question has the following reformulation:
    
\medskip
    \textbf{Open Question.} \emph{If length of a simple closed plane curve with curvature at most 1 is at least $4\pi$
then diameter is at least $4$.}
\medskip

    Also note that if diameter of $\gamma$ is at least 4, then according to theorem shortest possible curve is the one given on the figure 4.
    \begin{center}
        \begin{tikzpicture}
            \draw (0,0) circle (1);
            \draw (2,0) circle (1);

            \draw [black,ultra thick,domain=90:270] plot ({cos(\x)}, {sin(\x)});
            \draw [black,ultra thick,domain=90:-90] plot ({2+cos(\x)}, {sin(\x)});

            \draw [black,ultra thick] (0,1) -- (2,1);
            \draw [black,ultra thick] (0,-1) -- (2,-1);
        \end{tikzpicture}
        \captionof{figure}{\emph{Diameter is 4 and length is $2\pi+4$.}}
    \end{center}

   \medskip

    \parbf{Curves with length $4\pi$ and diameter $4$:}
Obvious example is the circle of radius 2. So, let us give some of the interesting examples. The figure 5 is one such.
\begin{center}
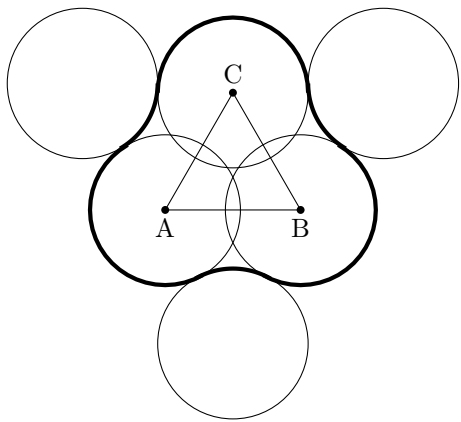


\begin{tikzpicture}
  \coordinate (A) at (0,0);
  \coordinate (B) at (1.8,0);
  \coordinate (C) at (0.9,1.56); 
  \coordinate (AB') at (0.9,-1.78);
  \coordinate (AC') at (-1.1,1.683);
  \coordinate (BC') at (2.9,1.683);
  \node[below] at (A) {A};
  \node[below] at (B) {B};
  \node[above] at (C) {C};
\fill[black]  (A) circle(1.5pt);
\fill[black]  (B) circle(1.5pt);
\fill[black]  (C) circle(1.5pt);
  \draw (A) -- (B) -- (C) -- cycle;

  \draw (A) circle (1);
  \draw (B) circle (1);
  \draw (C) circle (1);
  \draw (AB') circle (1);
  \draw (AC') circle (1);
  \draw (BC') circle (1);

   \draw [black,ultra thick,domain=0:-60] plot ({-1.1+cos(\x)}, {1.683+sin(\x)});
\draw [black,ultra thick,domain=60:120] plot ({0.9+cos(\x)}, {-1.78+sin(\x)});
   \draw [black,ultra thick,domain=180:240] plot ({2.9+cos(\x)}, {1.683+sin(\x)});
   \draw [black,ultra thick,domain=120:300] plot ({cos(\x)}, {sin(\x)});
\draw [black,ultra thick,domain=60:-120] plot ({1.8+cos(\x)}, {sin(\x)});
   \draw [black,ultra thick,domain=0:180] plot ({0.9+cos(\x)}, {1.56+sin(\x)});

 \end{tikzpicture}
\end{center}
\captionof{figure}{\emph{Three almost tangent unit circles.}}
\medskip
The circles in the picture are all unit and the triangle is equilateral. Moreover the circles inside the border have centers on the corners of the equilateral triangle. Our curve is the bordered one, call it $\gamma$. Let us call the length of a side of the triangle $2-l$. It is clear that if $l>0$ then no 2 disjoint open unit disks fit inside $\gamma$. As $l$ goes to $0$, length of $\gamma$ goes to $4\pi$, we leave it to the reader to calculate that, and at $l=0$ there are 3 disjoint open unit disks inside $\gamma$, and all of them touch each other.

    \hspace{0.5cm} Any curve of constant width 4 and curvature at most 1 are also examples to such curves. Also one
 may start with a curve of constant width 4 and curvature at most 1;
if it contains an arc of curvature 1/3 then it can be exchanged to a concatenation of 3 arcs of curvature 1 the same way as in the example above. Therefore we can produce huge variety of examples.

\medskip
\textbf{Acknowledgements} We wish to thank anonymous referee for his valuable comments and suggestions.

\bibliographystyle{plain}
\bibliography{refs}

\Addresses
\end{document}